\documentclass[12pt]{amsart}
\usepackage{amsmath, amsthm, amscd, amsfonts, amssymb, graphicx, color}
\usepackage{enumerate}
\usepackage[bookmarksnumbered, colorlinks, plainpages]{hyperref}
\usepackage{tikz}
\usepackage{comment}

\newtheorem{theorem}{Theorem}[section]
\newtheorem{lemma}[theorem]{Lemma}
\newtheorem{proposition}[theorem]{Proposition}

\newtheorem{corollary}[theorem]{Corollary}
\newtheorem{definition}[theorem]{Definition}
\newtheorem{example}[theorem]{Example}
\newtheorem{remark}[theorem]{Remark}
\newtheorem{conjecture}[theorem]{Conjecture}

\excludecomment{mysection}

\DeclareMathOperator{\co}{co}

\begin{document}
	
\title[Norming order units]{Normed linear spaces which are isometric to order unit spaces}
\author{Anil Kumar Karn}
	
\address{School of Mathematical Sciences, National Institute of Science Education and Research Bhubaneswar, An OCC of Homi Bhabha National Institute, P.O. - Jatni, District - Khurda, Odisha - 752050, India.}

\email{\textcolor[rgb]{0.00,0.00,0.84}{anilkarn@niser.ac.in}}

	
	
\begin{abstract}
	In this paper, we consider the linear direct sum of a real normed linear space with an order unit space and with a base normed space to obtain respectively a new order unit space and a new base normed space. As a consequence, we find that an $\ell_1$-space may be shown to be an order unit space. (However, not in the natural order.) Dually, an $\ell_{\infty}$-space may be shown to be a base normed space. To understand this aberration, we characterize the real normed linear spaces which are isometrically isomorphic to order unit spaces. We prove that other than the classical case of $\ell_{\infty}$, $\ell_1$ is also isometrically isomorphic to an order unit space whereas $\ell_2$ is not isometrically isomorphic to any order unit space.
\end{abstract}

\maketitle 

\section{Introduction} 

Alaoglu theorem roughly says that every normed linear space can be isometrically embedded into $C(K)$ for some suitable compact Haudorff space $K$. However, in case of an order unit space, this embedding enhances to an order embedding. Even otherwise, order units play a very significant role in functional analysis in general and operator algebras in particular. For example, they sit on the interface of the commutative and the non-commutative theory. Therefore, discovering more classes of order unit spaces is always desirable.

It is a folklore that we can adjoin an order unit to any real normed linear space. Probably, this construction is due to M. M. Day \cite[Theorem 1.6.1 (Alternative proof)]{Jam70}. Recently, in \cite{K21} the author studied the notion of absolute value in such spaces. Since the adjoining of an order unit to a real normed linear space is a one dimensional extension, such spaces also have a base normed structure. This aspect was studied by the author in \cite{K22}. In this paper, we generalize this construction. We consider the (linear) direct sum of a real normed linear space either with an order unit space or with a base normed space to construct a new order unit space and a new base normed space respectively. It is interesting to note that in the case of an order unit space the direct sum turns out to be an $\ell_1$-sum and in the case of a base normed space, it is an $\ell_{\infty}$-sum. As a consequence, we find that an $\ell_1$-space may be shown to be an order unit space. (However, not in the natural order.) Dually, an $\ell_{\infty}$-space may be shown to be a base normed space. 

Thus it will interesting to know under what condition a normed linear space is isometrically isomorphic to an order unit space. In this paper, we address this question. We introduce the notion of \emph{norming order units} in a normed linear space. We observe that the standard order unit in an order unit space is a norming order unit. More importantly, we show that if a norm one element in a normed linear space is a norming order unit, then it determines a (unique) order structure under which the space is an order unit space and the order unit norm is equal to the given norm. We examine some known ($\ell_{\infty}$) as well as some unknown ($\ell_1$) cases. We also find an example of a normed linear space that does not possess any norming order unit. 

\section{Adjoining normed linear spaces with ordered normed linear spaces} 

In this section, we consider the (linear) direct sum of an ordered linear space with a normed linear space and obtain a new ordered linear space under two suitable conditions. We begin with a base normed space which is easier to handle. 

Let $(V, V^+)$ be a real ordered vector space. A convex subset $B \subset V^+$ is called a \emph{base} of $V^+$, if $0 \notin B$ and for each $u \in V^+$ with $u \ne 0$, there is a unique $b \in B$ and a real number $k > 0$ such that $u = k b$. If $V^+$ is generating, then the set $K := \co (B \bigcup -B)$ is an absolutely convex and absorbing set in $V$ which determines a seminorm $\Vert\cdot\Vert_B$ in $V$. This is called the \emph{base seminorm} generated by $B$. When $\Vert\cdot\Vert_B$ is a norm, $(V, V^+, B)$ is called a \emph{base normed space}. For a detailed discussion on such spaces, one may refer to any standard book on ordered vector spaces. We follow \cite{A71,Jam70,WN73}. 
\begin{theorem}\label{basen}
	Let $(V, V^+, B)$ be a base normed space and let $X$ be any real normed linear space. Then $(V_X( := V \oplus_{\infty} X), V_X^+, B_X)$ is a base normed space for the cone $V_X^+ := \lbrace (u, x): u \in V^+ ~ \mbox{and} ~ \Vert u \Vert \ge \Vert x \Vert \rbrace$ and the base $B_X := B \times X_1$. (Here $X_1$ is the closed unit ball of $X$.) 
\end{theorem} 
\begin{proof}
	Note that $(0, 0) \in V_X^+$. Let $(u, x), (v, y) \in V_X^+$. Then $u, v \in V^+$ and we have $\Vert u \Vert \ge \Vert x \Vert$ and $\Vert v \Vert \ge \Vert y \Vert$. As $V^+$ is a cone, we have $u + v \in V^+$. Further, as the base norm is additive in $V^+$, we get 
	$$\Vert u + v \Vert = \Vert u \Vert + \Vert v \Vert \ge \Vert x \Vert + \Vert y \Vert \ge \Vert x + y \Vert.$$ 
	Thus $(u + v, x + y) \in V_X^+$. Now, if $k \in \mathbb{R}^+$, then $k u \in V^+$ and we have 
	$$\Vert k u \Vert = k \Vert u \Vert \ge k \Vert x \Vert = \Vert k x \Vert.$$
	Thus $k (u, x) \in V_X^+$ so that $V_X^+$ is a cone. Also, as $B$ and $X_1$ are convex subsets of $V^+$ and $X$ respectively, we get that $B_X$ is convex. Also, if $(u, x) \in B_X$, then $u \in B$ and $x \in X_1$. Thus $\Vert u \Vert = 1$ and $\Vert x \Vert \le 1$ so that $(u, x) \in V_X^+$ and subsequently, $B_X$ is a convex subset of $V_X^+$. We show that $B_X$ is a base for $V_X^+$. For this, let $(u, x) \in V_X^+$ with $(u, x) \ne (0, 0)$. Then $u \in V^+$ with $\Vert u \Vert \ge \Vert x \Vert$ so that $u \ne 0$. Since $B$ is a base for $V^+$, there is a unique $b \in B$ such that $u = \Vert u \Vert b$. Put $z = \Vert u \Vert^{-1} x$. Then $(b, z) \in B_X$ and $(u, x) = \Vert u \Vert (b, z)$. 
	Hence $B_X$ is a base for $V_X^+$. 
	
	By definition, the base (semi) norm corresponding to the base $B_X$ is gauge of $\co (B_X \cup - B_X)$. As $- X_1 = X_1$, we get that 
	$$\co (B_X \cup - B_X) = \co (B \cup - B) \times X_1 = V_1 \times X_1 = (V \oplus_{\infty} X)_1.$$ 
	Thus the base norm on $V_X$ determined by $B_X$ is given by 
	$$\Vert (v, x) \Vert_{\infty} := \max \lbrace \Vert v \Vert, \Vert x \Vert \rbrace$$
	for any $(v, x) \in V_X$. 
\end{proof} 
Now we turn to the dual notion. Let $(V, V^+)$ be a real ordered vector space and let $e \in V^+$. We say that $e$ is an \emph{order unit} for $V$, if for every $v \in V$ there exists $\lambda \in \mathbb{R}$ with $\lambda > 0$ such that $\lambda e \pm v \in V^+$. In this case, $e$ determines a semi-norm $\Vert\cdot\Vert_e$ given by 
$$\Vert v \Vert_e := \inf \lbrace \lambda > 0: \lambda e \pm v \in V^+ \rbrace$$ 
for all $v \in V$. 

We say that $V^+$ is \emph{proper}, if $V^+ \cap - V^+ = \lbrace 0 \rbrace$. Further, we say that $V^+$ is \emph{Archimedean}, if for any $v \in V$ with $\lambda u + v \in V^+$ for a fixed $u \in V$ and all $\lambda > 0$, we have $v \in V^+$. Now, an \emph{order unit space} is a real ordered vector space with an order unit $e$ such that $V^+$ is proper and Archimedean. (In this case, $\Vert\cdot\Vert_e$ is a norm on $V$ such that $V^+$ is norm-closed.) For details, we refer to \cite{A71,Jam70,WN73}, or for that matter, any standard book on ordered vector spaces. We quickly recall that the Banach dual of a base normed space is an order unit space and vice versa. 

We noticed that the adjoining a normed linear space to an order unit space is a little more complicated. Let $X$ be a real normed linear space and let $(V, V^+, e)$ be an order unit space. Then $V' \oplus_{\infty} X'$ is a base normed space as described in Theorem \ref{basen}. Thus the expected cone $V_X^+$ in $V\oplus_1 X$ should be described as 
$$V_X^+ = \lbrace (u, x) \in V \oplus_1 X: \langle (\phi, f), (u, x) \rangle \ge 0 ~ \mbox{for all} ~ \phi \in S(V) ~ \mbox{and} ~ f \in X_1' \rbrace.$$
Let $(u, x) \in V_X^+$. Then 
$$\inf \lbrace \phi(u): \phi \in S(V) \rbrace + \inf \lbrace f(x): f \in X_1' \rbrace \ge 0.$$ 
\begin{lemma}\label{inf}
	Let $X$ be a real normed linear space and let $(V, V^+, e)$ be an order unit space. 
	\begin{enumerate}
		\item For $x \in X$, we have 
		$$\inf \lbrace f(x): f \in X_1' \rbrace = - \Vert x \Vert.$$ 
		\item For $u \in V^+$, we have 
		$$\inf \lbrace \phi(u): \phi \in S(V) \rbrace = \Vert u \Vert - \left\Vert \Vert u \Vert e - u \right\Vert.$$
	\end{enumerate}
\end{lemma} 
\begin{proof}
	(1): Find $f_0 \in X_1'$ such that $\Vert x \Vert = f_0(x)$. Then $f(x) \le f_0(x)$ for all $f \in X_1'$. Since $- f \in X_1'$ whenever $f \in X_1'$, we get that 
	$$\inf \lbrace f(x): f \in X_1' \rbrace = - f_0(x) = - \Vert x \Vert.$$ 
	(2): Since $S(V)$ is weak*-compact, there exists $\phi_0 \in S(V)$ such that $\phi_0(u) \le \phi(u)$ for all $\phi \in S(V)$. Thus 
	$$\phi(\Vert u \Vert e - u) = \Vert u \Vert - \phi(u) \le \Vert u \Vert - \phi_0(u) = \phi_0(\Vert u \Vert e - u)$$
	for all $\phi \in S(V)$. Therefore, 
	$$\left\Vert \Vert u \Vert e - u \right\Vert = \phi_0(\Vert u \Vert e - u) = \Vert u \Vert - \phi_0(u),$$, 
	or equivalently, 
	$$\inf \lbrace \phi(u): \phi \in S(V) \rbrace = \phi_0(u) = \Vert u \Vert - \left\Vert \Vert u \Vert e - u \right\Vert.$$
\end{proof}
We deduce that the expected cone $V_X^+$ should thus be given by 
$$V_X^+ = \lbrace (u, x) \in V \oplus_1 X: u \in V^+ ~ \mbox{and} ~ \Vert x \Vert \le \Vert u \Vert - \left\Vert \Vert u \Vert e - u \right\Vert \rbrace.$$ 
The condition given in $V_X^+$ can be presented in a simpler form. 
\begin{lemma}
	Let $(V, V^+, e)$ be an order unit space. Then for $u \in V$ and $k \in \mathbb{R}^+$, following two conditions are equivalent: 
	\begin{enumerate}
		\item $u \in V^+$ and $k \le \Vert u \Vert - \left\Vert \Vert u \Vert e - u \right\Vert \rbrace$. 
		\item $k e \le u$.
	\end{enumerate}
\end{lemma}
\begin{proof}
	First we assume that $u \in V^+$ and $k \le \Vert u \Vert - \left\Vert \Vert u \Vert e - u \right\Vert$. Then $\left\Vert \Vert u \Vert e - u \right\Vert \le \Vert u \Vert - k$. Thus $0 \le \Vert u \Vert e - u \le (\Vert u \Vert - k) e$ so that $k e \le u$. 
	
	Conversely, we assume that $k e \le u$. Then $u \in V^+$ with $k \le \Vert u \Vert$. Since $u \le \Vert u \Vert e$, we get $0 \le \Vert u \Vert e - u \le \Vert u \Vert e - k e$. Thus $\left\Vert \Vert u \Vert e - u \right\Vert \le \Vert u \Vert - k$, whence $k \le \Vert u \Vert - \left\Vert \Vert u \Vert e - u \right\Vert$. 
\end{proof}
\begin{theorem}\label{adjno}
	Let $(V, V^+, e)$ be an order unit space and $X$ is a real normed linear space. Then $(V \oplus_1 X, V_X^+, e_X)$ is an order unit space where $V_X^+ = \lbrace (u, x): \Vert x \Vert e \le u \rbrace$ and $e_X := (e, 0)$.
\end{theorem} 
\begin{proof}
	Let $(u, x) \in V_X^+$ and let $\alpha $ be a positive real number. Then $\Vert x \Vert e \le u$ so that $\Vert \alpha x \Vert e \le \alpha u$. Thus $\alpha (u, x) \in V_X^+$. Next, let $(u, x), (v, y) \in V_X^+$. Then $\Vert x \Vert e \le u$ and $\Vert y \Vert e \le v$. Thus $\Vert x + y \Vert e \le (\Vert x \Vert + \Vert y \Vert) e \le u + v$ so that $(u, x) + (v, y) \in V_X^+$ as $u + v \in V^+$. Hence $V_X^+$ is a cone. 
	
	Next, let $\pm (u, x) \in V_X^+$. Then $\Vert x \Vert e \le u$ and $\Vert - x \Vert e \le - u$. Thus $\pm u \in V^+$. Since $V^+$ is proper, we get $u = 0$ and consequently, $x = 0$. Thus $V_X^+$ is also proper.
	
	Now, we show that $e_X$ is an order unit for $V \oplus_1 X$. For this, let $(v, x) \in V \oplus_1 X$. Put $k = \Vert v \Vert + \Vert x \Vert$. Since $\Vert v \Vert e \pm v \ge 0$, we get $k e \pm v \ge \Vert x \Vert e$. Thus $k e_X \pm (u, x)\in V_X^+$. Therefore, $e_X$ is an order unit for $V \oplus_1 X$. We also notice that $\Vert (v, x) \Vert_o \le \Vert (v, x) \Vert_1$ for all $(v, x) \in V \oplus_1 X$ where $\Vert\cdot\Vert_o$ is the corresponding order unit (semi)norm on $V \oplus_1 X$.
	
	Next, we prove that $V_X^+$ is Archimedean. Let $(v, x) \in V \oplus_1 X$ be such that $k e_X + (v, x) \in V_X^+$ for all $k > 0$. That is, $(k e + v, x) \in V_X^+$ for all $k > 0$. Then $\Vert x \Vert e \le k e + v$ for all $k > 0$. Now by the Archimedean property in $V^+$, we get $\Vert x \Vert e \le v$, that is, $(v, x) \in V_X^+$. Therefore, $V_X^+$ is Archimedean. 
	
	Hence $(V \oplus_1 X, V_X^+, e_X)$ is an order unit space. Also, we have that the corresponding order unit norm $\Vert \cdot \Vert_o \le \Vert \cdot \Vert_1$. We prove their equality. Let $(v, x) \in V \oplus_1 X$. Assume that $\lambda e_X \pm (v, x) \in V_X^+$, that is, $(\lambda e \pm v, \pm x) \in V_X^+$ for some $\lambda > 0$. Then $\Vert x \Vert e \le \lambda e \pm v$. Thus $\pm v \le (\lambda - \Vert x \Vert) e$. It follows that $\Vert v \Vert \le \lambda - \Vert x \Vert$, that is $\Vert (v, x) \Vert_1 \le \lambda$. Hence $\Vert (v, x) \Vert_1 \le \Vert (v, x) \Vert_o$. This completes the proof. 
\end{proof}

Now we describe the order unit space $(\ell_1, e_1)$ as an order unit space by adjoining a normed linear space to an order unit space.
\begin{example}
	Consider the order unit space $(\mathbb{R}, 1)$ and the normed linear space $\mathbb{R}$ in their natural setting. Then by Theorem \ref{adjno}, $\ell_1^2$ is an order unit space with the cone $(\ell_1^2)_{e_1^2}^+ := \lbrace (\alpha, \beta): \alpha \ge \vert \beta \vert \rbrace$ and the order unit $e_1^2 := (1, 0)$. Next, consider the resulting order unit space $(\ell_1^2, (\ell_1^2)_{e_1^2}^+, e_1^2)$ and again the naturally normed $\mathbb{R}$. Again invoking Theorem \ref{adjno}, we now obtain the order unit space $(\ell_1^3, (\ell_1^3)_{e_1^3}^+, e_1^3)$ where $e_1^3 = (1, 0, 0)$ and $(\ell_1^3)_{e_1^3}^+ := \left\lbrace (\alpha, \beta, \gamma): (\alpha, \beta) \ge \vert \gamma \vert e_1^2 ~ \mbox{in} ~ (\ell_1^2)_{e_1^2}^+ \right\rbrace$. If $(\alpha, \beta, \gamma) \in (\ell_1^3)_{e_1^3}^+$, then $\vert \gamma \vert e_1^2 \le (\alpha, \beta)$ in $(\ell_1^2)_{e_1^2}^+$ so that $(\alpha - \vert \gamma \vert, \beta) \in (\ell_1^2)_{e_1^2}^+$. Thus $\vert \beta \vert \le \alpha - \vert \gamma \vert$, that is, $\Vert (\beta, \gamma) \Vert_1 = \vert \beta \Vert + \vert \gamma \vert \le \alpha$. Now it follows that $(\ell_1^3, (\ell_1^3)_{e_1^3}^+, e_1^3)$ is also the oder unit obtained by adjoining the normed linear space $\ell_1^2$ to the order unit space $(\mathbb{R}, 1)$. Also following the induction on $n$, we may conclude that for each $n \in \mathbb{N}$, $(\ell_1^{n+1}, (\ell_1^{n+1})_{e_1^{n+1}}^+, e_1^{n+1})$ can be obtained as the order unit space by adjoining $\mathbb{R}$ to $(\mathbb{R}, 1)$ successively $n$-times or simply adjoining $\ell_1^n$ to $(\mathbb{R}, 1)$. As a result of induction on $n$, we can also conclude that $(\ell_1, (\ell_1)_{e_1^{n+1}}^+, e_1)$ can be obtained as the order unit space by adjoining $\mathbb{R}$ to $(\mathbb{R}, 1)$ successively countably many times or simply adjoining $\ell_1$ to $(\mathbb{R}, 1)$.
\end{example}
\begin{remark}
	A dual statement may be made about $\ell_{\infty}$ as a base normed space which we leave to the readers. Furthermore, when a (real) Hilbert space is adjoined to the order unit space $(\mathbb{R}, 1)$, we get spin factors \cite[Section 6.1]{HS84}. (See also \cite{K21} where we have generalized this idea.) 
\end{remark} 
\begin{remark}
	By \cite[Lemma 5.1]{K24}, every $2$-dimensional order unit space is isometrically isomorphic to $\ell_{\infty}^2$. Here we construct a map to demonstrate that as a real normed linear space, $\ell_1^2$ is isometrically isomorphic to the order unit space $\ell_{\infty}^2$. 
	
	Consider $\mathbb{R}$ as a real normed linear space with the standard norm as the modulus and also consider $(\mathbb{R}, 1)$ as the order unit space with its natural cone. Then $\mathbb{R} \oplus_1 \mathbb{R} \cong \ell_1^2$ and by Theorem \ref{adjno}, it is an order unit space with its cone $C = \lbrace (\alpha, \beta): \vert \beta \vert \le \alpha \rbrace$ and the order unit $e_1 := (1, 0)$. Now define $\psi: \ell_1^2 \to \ell_{\infty}^2$ given by $\psi(\alpha, \beta) = (\alpha + \beta, \alpha - \beta)$ for all $\alpha, \beta \in \mathbb{R}$. Then $\psi$ is a linear, isometric isomorphism. Also $\psi(C) = \lbrace (x, y): x, y \in [0, \infty) \rbrace$ and $\psi(e_1) = (1, 1)$.
\end{remark}

\section{Norming order unit} 

When we think of an order unit space, generally $\ell_{\infty}$-type spaces appear in our minds. However, we have noted in the previous section that even $\ell_1$ can be shown to be an order unit space, albeit not in its natural order. This leads to a curiosity: which normed linear space can be seen as order unit space. In this section, we try to address this question. Let us begin with a notion.
\begin{definition}
	Let $X$ be a non-zero real normed linear space and let $e \in X$ with $\Vert e \Vert = 1$. We say that $e$ is a \emph{norming order unit} for $X$, if $\Vert 2 x - \Vert x \Vert e \Vert \le \Vert x \Vert$ whenever $\Vert 2 x - \lambda e \Vert \le \lambda$ for some $\lambda \ge \Vert x \Vert$.
\end{definition}
Note that $- e$ is a norming order unit for $X$ whenever $e$ is a norming order unit for $X$. More generally, if $T: X \to X$ is a surjective linear isometry, then $T(e)$ is a norming order unit for $X$ whenever $e$ is a norming order unit for $X$. The next result justifies the nomenclature ``norming order unit''.
\begin{proposition}\label{pos}
	Let $(V, e)$ be an order unit space and let $u \in V$. Then the following statements are equivalent: 
	\begin{enumerate}
		\item $u \in V^+$.
		\item $\Vert 2 u - \lambda e \Vert \le \lambda$ for any $\lambda \ge \Vert u \Vert$.
		\item $\Vert 2 u - \lambda e \Vert \le \lambda$ for some $\lambda \ge \Vert u \Vert$. 
		\item $\Vert 2 u - \Vert u \Vert e \Vert = \Vert u \Vert$.
	\end{enumerate}
\end{proposition} 
\begin{proof}
	Let $u \in V^+$ and assume that $\lambda \ge \Vert u \Vert$. Then $0 \le u \le \Vert u \Vert e \le \lambda e$. Thus $\lambda e \pm (2 u - \lambda e) \in V^+$ and we have $\Vert 2 u - \lambda e \Vert \le \lambda$. Conversely, assume that $\Vert 2 u - \lambda e \Vert \le \lambda$ for some $\lambda \ge \Vert u \Vert$. Then $\lambda e \pm (2 u - \lambda e) \in V^+$ so that $u \in V^+$. Hence (1), (2) and (3) are equivalent. Also then (4) implies (1). Finally, assume that $u \in V^+$ and let $\Vert 2 u - \Vert u \Vert e \Vert = \alpha$. Then $\alpha e \pm (2 u - \Vert u \Vert e) \in V^+$. Thus $0 \le u \le \frac 12 (\Vert u \Vert + \alpha) e$. It follows that $\Vert u \Vert \le \alpha$. By (2), we have $\alpha \le \Vert u \Vert$ which completes the proof. 
\end{proof} 
We show that whenever a real normed linear space possesses a norming order unit, it is isometric to an order unit space. 
\begin{theorem}\label{nou}
	Let $(X, \Vert\cdot\Vert)$ be a real normed linear space and assume that $e \in X$ with $\Vert e \Vert = 1$ be a norming order unit for $X$. Put 
	$$X_e^+ = \lbrace x \in X: \Vert 2 x - \Vert x \Vert e \Vert \le \Vert x \Vert \rbrace.$$ 
	Then $X_e^+$ is a proper and Archimedean cone in $X$ and $e$ is an order unit for $(X, X_e^+)$ so that $(X, X_e^+, e)$ is an order unit space. Moreover, $e$ determines $\Vert\cdot\Vert$ as an order unit norm on $X$.
\end{theorem} 
\begin{proof}
	We have $0, e \in X_e^+$ so that $X_e^+ \ne \emptyset$. Let $x \in X_e^+$ and $\alpha \ge 0$. Then $\Vert 2 x - \Vert x \Vert e \Vert \le \Vert x \Vert$. Thus 
	$$\Vert 2 \alpha x - \Vert \alpha x \Vert e \Vert = \alpha \Vert 2 x - \Vert x \Vert e \Vert \le \alpha \Vert x \Vert = \Vert \alpha x \Vert$$ 
	so that $\alpha x \in X_e^+$. Next, let $x, y \in X_e^+$. Then $\Vert 2 x - \Vert x \Vert e \Vert \le \Vert x \Vert$ and $\Vert 2 y - \Vert y \Vert e \Vert \le \Vert y \Vert$. Thus 
	\begin{eqnarray*}
		\Vert 2 (x + y) - (\Vert x \Vert + \Vert y \Vert) e \Vert &\le& \Vert 2 x - \Vert x \Vert e \Vert + \Vert 2 y - \Vert y \Vert e \Vert \\ 
		&\le& \Vert x \Vert + \Vert y \Vert.
	\end{eqnarray*}
	Since $e$ is a norming order unit for $X$, we conclude that $\Vert 2 (x + y) - \Vert x + y \Vert e \Vert \le \Vert x + y \Vert$. Thus $x + y \in X_e^+$ and $X_e^+$ is a cone. We observe that if $x \in X_e^+$, then for any $\lambda \ge \Vert x \Vert$, we have $\Vert 2 x - \lambda e \Vert \le \lambda$. In fact, 
	\begin{eqnarray*}
		\Vert 2 x - \lambda e \Vert &=& \Vert (2 x - \Vert x \Vert e) - (\lambda - \Vert x \Vert) e \Vert \\ 
		&\le& \Vert 2 x - \Vert x \Vert e \Vert + \Vert (\lambda - \Vert x \Vert) e \Vert \\ 
		&\le& \Vert x \Vert + \lambda - \Vert x \Vert = \lambda.
	\end{eqnarray*} 
	Hence, $x \in X_e^+$ if and only if $\Vert 2 x - \lambda e \Vert \le \lambda$ for any $\lambda \ge \Vert x \Vert$. 
	
	Now, we show that $e$ is an order unit for $(X, X_e^+)$. Let $x \in X$. Then $\Vert (\Vert x \Vert e \pm x) - \Vert x \Vert e \Vert = \Vert x \Vert$. Since $\Vert \Vert x \Vert e \pm x \Vert \le 2 \Vert x \Vert$, we conclude that $\Vert x \Vert e \pm x \in X_e^+$ for all $x \in X$. Thus $e$ is an order unit for $X$. We also deduce that $\Vert x \Vert_e \le \Vert x \Vert$ for all $x \in X$ where $\Vert\cdot\Vert_e$ is the order unit (semi)norm on $X$. We prove the equality. Let $x \in X$ and $\lambda \ge 0$ be such that $\lambda e \pm x \in X_e^+$. Since $\Vert \lambda e \pm x \Vert \le \lambda + \Vert x \Vert$, we have 
	$$\Vert 2 (\lambda e \pm x) - (\lambda + \Vert x \Vert) e \Vert \le \lambda + \Vert x \Vert.$$ 
	In other words, $\Vert 2  x \pm (\lambda - \Vert x \Vert) e \Vert \le \lambda + \Vert x \Vert$. Thus 
	$$4 \Vert x \Vert \le \Vert 2  x + (\lambda - \Vert x \Vert) e \Vert + \Vert 2  x - (\lambda - \Vert x \Vert) e \Vert \le 2 (\lambda + \Vert x \Vert)$$ 
	so $\Vert x \Vert \le \lambda$. Hence $\Vert x \Vert \le \Vert x \Vert_e$ and consequently, $\Vert x \Vert = \Vert x \Vert_e$ for every $x \in X$. Further, if $\pm x \in X_e^+$ for some $x \in X$, then as above, we get $\Vert x \Vert \le 0$. Thus $x = 0$ so that $X_e^+$ is proper. 
	
	Finally, we prove that $X_e^+$ is Archimedean. Assume that $x \in X$ be such that $x + \lambda e \in X_e^+$ for all $\lambda > 0$. Since $\Vert x + \lambda e \Vert \le \Vert x \Vert + \lambda$, we have $\Vert 2 (x + \lambda e) - (\Vert x \Vert + \lambda) e \Vert \le \Vert x \Vert + \lambda$ for all $\lambda > 0$. Letting $\lambda \to 0^+$, we get $\Vert 2 x - \Vert x \Vert e \Vert \le \Vert x \Vert$. Thus $x \in X_e^+$. That is, $X_e^+$ is Archimedean.
\end{proof}
The following observation is a direct consequence of Proposition \ref{pos} which is the converse of Theorem \ref{nou}.
\begin{corollary}
	Let $(V, e)$ be an order unit space. Then $e$ is a norming order unit for $V$ with $V_e^+ = V^+$.
\end{corollary} 
\begin{corollary}
	A strictly convex real normed linear space of dimension more than $1$ does not possess any norming order unit.
\end{corollary}
\begin{proof}
	Let $V$ be a strictly convex real normed linear space of dimension more than $1$. Assume to the contrary that $e \in V$ with $\Vert e \Vert = 1$ is a norming order unit for $V$. Then $(V, V_e^+, e)$ is an order unit space such that the norm on $V$ is the order unit norm determined by $e$. Here $V_e^+ = \lbrace u \in V: \Vert 2 u - \Vert u \Vert e \Vert \le \Vert u \Vert \rbrace$. Let $u \in V_e^+$ with $\Vert u \Vert = 1$ and $u \ne e$. It follows from Proposition \ref{pos} that $\Vert 2 u - e \Vert = 1$. Put $v := 2 u - e$. Then $\Vert v \Vert = 1$ with $v \ne e$. Since $V$ is strictly convex, we must have $\Vert \frac 12 (e + v) \Vert < 1$. This leads to a contradiction as $\frac 12 (e + v) = u$. 
\end{proof} 
In \cite{K24}, the author has shown that the skeleton $S_V := \lbrace u \in V^+: \Vert u \Vert = 1 = \Vert e - u \Vert \rbrace$ of an order unit space $(V, V^+, e)$ describes the whole space. So it may be interesting to describe skeletal elements (that is, the elements in $S_V$) in terms of norming order units independently. We shall use $\infty$-orthogonality for this purpose. Let us recall that $x, y \in X$ in a real Banach space $X$ are said to \emph{orthogonal} to each other, if $\Vert x + k y \Vert = \max \lbrace \Vert x \Vert, \Vert k y \Vert \rbrace$ for all real number $k$ \cite{K14}.
\begin{proposition}
	Let $X$ be a real normed linear space, $e \in X$ with $\Vert e \Vert = 1$ be a norming order unit for $X$ and assume that $u \in X$. Then $u$ is a skeletal element in the order unit space $(X, X_e^+, e)$ if and only if $\Vert u \Vert = \Vert e - u \Vert$ and $\Vert 2 u - e \Vert = 1$. 
\end{proposition} 
\begin{proof}
	First, we assume that $u$ is a skeletal element of $(X, X_e^+, e)$. That is, $u \in X_e^+$ and $\Vert u \Vert = 1 = \Vert e - u \Vert$. Then $e - u \in X_e^+$.  Since $1 = \Vert e \Vert = \Vert u + (e - u) \Vert$, we have $u \perp_{\infty} (e - u)$ by \cite[Theorem 3.3]{K14}. Thus
	$$\Vert 2 u - e \Vert = \Vert u - (e - u) \Vert = \max \lbrace \Vert u \Vert, \Vert e - u \Vert \rbrace = 1.$$ 
	Conversely, we assume that $\Vert u \Vert = \Vert e - u \Vert$ and $\Vert 2 u - e \Vert = 1$. Since $e$ is the norming order unit for $X_e^+$, we have $e \pm (2 u - e) \in X_e^+$, or equivalently, $u, e - u \in X_e^+$. Thus $\Vert u \Vert = \Vert e - u \Vert \le 1$. Also then 
	$$1 = \Vert 2 u - e \Vert = \Vert u - (e - u) \Vert \le \max \lbrace \Vert u \Vert, \Vert e - u \Vert \rbrace = \Vert u \Vert$$ 
	so that $\Vert u \Vert = \Vert e - u \Vert = 1$. Thus $u$ is a skeletal element in the order unit space $(X, X_e^+, e)$. 
\end{proof}
\subsection{Some examples.} 
\begin{example}
	{\bf Norming order units in $\ell_{\infty}$.} Since the constant sequence $\langle 1\rangle$ is the norming order unit for $\ell_{\infty}$ together with its natural cone, and since $\langle \alpha_n \rangle \mapsto \langle \epsilon_n \alpha_n \rangle$ is a surjective isometry in $\ell_{\infty}$ where $\epsilon_n \in \lbrace 1, - 1 \rbrace$, we deduce that $\langle \epsilon_n \rangle$ is also a norming order unit for $\ell_{\infty}$ for every choices of $\epsilon_n$. We show that these are the only norming order units in $\ell_{\infty}$. 
	
	Let $e_0 = \langle \alpha_n \rangle \in \ell_{\infty}$ with $\Vert \langle \alpha_n \rangle \Vert_{\infty} = 1$. Without any loss of generality, we also assume that $\alpha_n \ge 0$ for every $n$. Thus $\sup_{n \in \mathbb{N}} \alpha_n = 1$. Let that $\alpha_{n_0} < 1$ for some $n_0 \in \mathbb{N}$. Put $\lambda_0 = \frac{2}{1 + \alpha_{n_0}}$. Then $\lambda_0 > 1$. Consider $e_{n_0} \in \ell_{\infty}$ with $1$ at $n_0$-th place and $0$ elsewhere. Then $\Vert e_{n_0} \Vert_{\infty} = 1$. Also 
	$$\Vert 2 e_{n_0} - \lambda_0 e_0 \Vert_{\infty} = \max \lbrace \vert 2 - \lambda_0 \alpha_{n_0} \vert, \sup_{n \ne n_0} \alpha_n \rbrace = \lambda_0.$$ 
	However,
	$$\Vert 2 e_{n_0} - e_0 \Vert_{\infty} = \max \lbrace \vert 2 - \alpha_{n_0} \vert, \sup_{n \ne n_0} \alpha_n \rbrace = 2 - \alpha_{n_0} > 1 = \Vert e_{n_0} \Vert_{\infty}.$$ 
	Thus $e_0$ fails to be a norming order unit for $\ell_{\infty}$. 
\end{example} 
\begin{example}\label{l1}
	{\bf Norming order units in $\ell_1$.} Let $e_0 = \langle \alpha_n \rangle \in \ell_1$ with $\Vert \langle \alpha_n \rangle \Vert_1 = 1$. Without any loss of generality, we can again assume that $\alpha_n \ge 0$ for every $n$. Then $\sum_{n \in \mathbb{N}} \alpha_n = 1$. 
	
	Case 1. Assume that $0 < \alpha_{n_0} < 1$ for some $n_0$. Consider $e_{n_0} \in \ell_1$ with $1$ at $n_0$-th place and $0$ elsewhere. Then $\Vert e_{n_0} \Vert_1 = 1$. Also 
	\begin{eqnarray*}
		\left\Vert 2 e_{n_0}  - \frac{2}{\alpha_{n_0}} e_0 \right\Vert_1 &=& \left(\frac{2}{\alpha_{n_0}} \right) \Vert \alpha_{n_0} e_{n_o} - e_0 \Vert_1 \\ 
		&=& \frac{2}{\alpha_{n_0}} \left(\sum_{n \ne n_0} \alpha_n \right) \\ 
		&=& \frac{2 (1 - \alpha_{n_0})}{\alpha_{n_0}} \\ 
		&<& \frac{2}{\alpha_{n_0}}.
	\end{eqnarray*} 
	However 
	\begin{eqnarray*}
		\left\Vert 2 e_{n_0}  - e_0 \right\Vert_1 
		&=& (2 - \alpha_{n_0}) + \left(\sum_{n \ne n_0} \alpha_n \right) \\ 
		&=& (2 - \alpha_{n_0}) + (1 - \alpha_{n_0}) \\ 
		&>& 1.
	\end{eqnarray*} 
	Thus in this case, $e_0$ fails to be a norming order unit for $\ell_1$. So we assume now that $\alpha_n \in \lbrace 0, 1 \rbrace$.
	
	Case 2. Assume that $\alpha_{n_0} = 1$ for some $n_0$. Then $\alpha_n = 0$, if $n \ne n_0$ so that $e_0 = e_{n_0}$. We show that $e_0$ is a norming order unit for $\ell_1$. Let $u = \langle x_n \rangle \in \ell_1$ with $\Vert u \Vert_1 = 1$. Let $\lambda \ge 1$ be such that $\Vert 2 u - \lambda e_0 \Vert_1 \le \lambda$ or equivalently, $\vert 2 x_{n_0} - \lambda \vert + 2 \sum_{n \ne n_0} \vert x_n \vert \le \lambda$. Since $\sum_{n \ne n_0} \vert x_n \vert = 1 - \vert x_{n_0} \vert$, we have $\vert 2 x_{n_0} - \lambda \vert \le \lambda  + 2 \vert x_{n_0} \vert - 2$. In other words, 
	$$- (\lambda  + 2 \vert x_{n_0} \vert - 2) \le 2 x_{n_0} - \lambda \le \lambda  + 2 \vert x_{n_0} \vert - 2,$$
	whence $1 - \vert x_{n_0} \vert \le x_{n_0} \le \lambda + \vert x_{n_0} \vert - 1$. Thus $x_{n_0} \ge 0$ and consequently, $x_{n_0} \ge \frac 12$. Therefore,  
	$$\Vert 2 u - e_0 \Vert_1 = (2 x_{n_0} - 1) + 2 \sum_{n \ne n_0} \vert x_n \vert = (2 x_{n_0} - 1) + 2 (1 - x_{n_0}) = 1.$$ 
	Hence $e_{n_0}$ is a norming order unit for $\ell_1$. In this way, we may conclude that $\lbrace \pm e_n: n \in \mathbb{N} \rbrace$ is the set of all norming order units for $\ell_1$. 
\end{example} 
We end the paper with the following conjecture.
\begin{conjecture}
	Let $(V, V^+, e)$ be an order unit space and assume that $u \in V$ with $\Vert u \Vert = 1$ is a norming order unit. Then there exists a surjective linear isometry $T: V \to V$ such that $T(e) = u$. 
\end{conjecture}
\thanks{
	{\bf Acknowledgements:} 
	The author was partially supported by Science and Engineering Research Board, Department of Science and Technology, Government of India sponsored  Mathematical Research Impact Centric Support project (reference no. MTR/2020/000017).}

\end{document}